\documentclass[12pt,reqno]{amsart}
\usepackage[margin=2.5 cm]{geometry}
\usepackage{amsfonts,amssymb,amsthm, mathtools}
\usepackage{amsmath}
\usepackage{datetime}
\usepackage{tabularx}
\everymath{\displaystyle}
\usepackage{xcolor}
\usepackage{enumerate,physics}
\usepackage{multirow}
\usepackage{float}
\everymath{\displaystyle}

\definecolor{fgreen}{RGB}{44,144, 14}

\numberwithin{equation}{section} 
\newtheorem{theorem}{Theorem}[section] 
\newtheorem{proposition}[theorem]{Proposition} 
 
\newtheorem{lemma}[theorem]{Lemma} 
\theoremstyle{definition}
\newtheorem{definition}[theorem]{Definition}

\usepackage[backref, colorlinks=true,citecolor=cyan,  urlcolor=blue,linkcolor=blue]{hyperref}

\def\R{\mathbb R}

\def\C{\mathbb C}

\def\H{\mathbb H}

\def\N{\mathbb N}
\def\Z{\mathbb Z}

\def\Q{\mathbb Q}

\def\R{\mathbb R}

\def\P{\mathbb P}

\def \x {{\bf x}}

\def\R{\mathbb {R}}
\def\C{\mathbb {C}}
\def\N{\mathbb {N}}
\def\H{\mathbb {H}}
\def\Z{\mathbb {Z}}

\def\CP{\mathbb {C} \mathrm{P}}

\newcommand{\defref}[1]{Definition~\ref{#1}}

\newcommand{\lemref}[1]{Lemma~\ref{#1}}

\begin{document}

	\title[Kulkarni limit sets for cyclic quaternionic projective groups]{Kulkarni limit sets for cyclic quaternionic projective groups}
	\author[S.  Dutta, K. Gongopadhyay and R. Mondal]{Sandipan Dutta, Krishnendu Gongopadhyay and 
		Rahul Mondal}
	
	\address{Mizoram University,
		Tanhril, Aizawl 796004, Mizoram, India}
	\email{sandipandutta98@gmail.com, mzut330@mzu.edu.in}
	
	\address{Indian Institute of Science Education and Research (IISER) Mohali,
		Knowledge City,  Sector 81, S.A.S. Nagar 140306, Punjab, India}
	\email{krishnendu@iisermohali.ac.in}

	\address{Indian Institute of Science Education and Research (IISER) Mohali,
		Knowledge City,  Sector 81, S.A.S. Nagar 140306, Punjab, India}
	\email{canvas.rahul@gmail.com}

	
	\date{ @\currenttime , \today}
	\begin{abstract}

    We consider the natural action of the quaternionic projective linear group $\mathrm{PSL}(n+1,\mathbb{H})$ on the quaternionic projective space $\mathbb{P}^n_{\mathbb{H}}$. We compute the Kulkarni limit sets for the cyclic subgroups of $\mathrm{PSL}(n+1,\mathbb{H})$. 
	\end{abstract}
    	\maketitle 

	\section{Introduction} 

    One way to study discrete subgroups of Lie groups is through their dynamical properties. A notable example involves Kleinian groups—discrete subgroups of $\mathrm{PSL}(2, \mathbb{C})$—which act on the Riemann sphere via Möbius transformations. Their orbits, if infinite, accumulate on a boundary-invariant subset called the \emph{limit set} \cite{cg,Mas}, with the group acting properly discontinuously on its complement. The properties of the limit sets carry a striking analogy with the behavior of rational maps of the Riemann sphere; cf. \cite{ctm}. 

A natural question in higher-dimensional dynamics concerns the behavior of discrete subgroups of ${\rm PSL}(n, \C)$. With the motivation to build a connection between the dynamics of projective automorphisms of $\CP^n$ and the discrete subgroups of ${\rm PSL}(n, \C)$, Seade and Verjovsky \cite{sv99},  \cite{sv01} have initiated investigation of the `Complex Kleinian groups', that is, discrete subgroups of complex projective linear transformations. A central object in this study is the Kulkarni limit set, introduced by Kulkarni in \cite{KUL}, who proposed a domain of discontinuity for any locally compact group acting on a locally compact Hausdorff space. Following the foundational work of Seade and Verjovsky, the study of Kulkarni limit sets for discrete subgroups of ${\rm PSL}(n, \mathbb{C})$, particularly in the case $n = 3$, has been further developed by Barrera, Cano, Navarrete, and others; see, for instance, the monograph \cite{cns} and the articles \cite{bcns, can, can2, TR, nav}.

To generalize this framework to the quaternionic setting, we consider discrete subgroups of ${\rm PSL}(n, \mathbb{H})$. The group ${\rm PSL}(2, \mathbb{H})$ corresponds to the isometry group of five-dimensional real hyperbolic space. This naturally leads to the intriguing question of understanding higher-rank projective linear groups over the quaternions. In \cite{dgl}, Kulkarni limit sets for cyclic subgroups of ${\rm PSL}(3, \mathbb{H})$ were classified. The goal of this paper is to extend that classification to Kulkarni limit sets of cyclic subgroups of ${\rm PSL}(n, \mathbb{H})$ for arbitrary \(n\), thereby generalizing the results of \cite{dgl}. 

It is worth noting that Kulkarni limit sets for cyclic subgroups of complex projective transformations in ${\rm PSL}(n, \mathbb{C})$ were computed by Cano et al.\ in \cite{cano_main}. Inspired by their approach, we adopt a similar methodology in this work. However, due to the non-commutativity of the quaternions, the computations require a  different treatment. 
  
   Since ${\rm SL}(n+1, \mathbb{H})$ is a double cover of ${\rm PSL}(n+1, \mathbb{H})$ by $\{\pm \mathrm{I}_n\}$, we often lift elements from the projective group to ${\rm PSL}(n+1, \mathbb{H})$ and consider their matrix representations. 

The limit sets will be described in the following table.  Before that, we first recall the definition of the Kulkarni limit sets.

	\begin{definition} \cite{KUL}
		Let $P_X=\{A_\beta:\; \beta \in B\}$ be a family of subsets of $X$ (where $B$ is an infinite indexing set). A point $p$ is called a \textit{cluster point} of $P_X$ if for every neighbourhood $N(p)$ of $p$, the  subset of the indexing set $B$ defined as $\{\beta\in B:\; N(p)\cap A_\beta\neq \emptyset\}$ is infinite.
	\end{definition}
	Consider the natural action of a subgroup  $G$  of ${\rm SL}(n+1, \H)$ on the $n$-dimensional quaternionic projective space $X= \P_{\H}^n$.  Consider the following three sets
	\begin{enumerate}[(a)]
		\item $L_0(G):=$ the closure of the set of points of $X$ which have an infinite isotropy group,
		\item $L_1(G):=$ the closure of the cluster points of orbits of points in $X\setminus L_0(G)$, and
		\item $L_2(G):=$ the closure of the cluster points of $\{g(K)\}_{g\in G}$, where $K$ runs over all the compact subsets of $X\setminus \{L_0(G)\cup L_1(G)\}$.
	\end{enumerate}
	These sets are called the Kulkarni sets.
	\begin{definition}\cite{KUL} \label{def-kul-limitset}
		With $G$ and $X$ as defined above, the \textit{Kulkarni limit set} of $G$ is
		\begin{equation}
			\Lambda(G):=L_0(G)\cup L_1(G)\cup L_2(G).
		\end{equation}
		The \textit{Kulkarni domain of discontinuity} of $G$ is defined as $\Omega(G)=X\setminus \Lambda(G)$.
	\end{definition}
	
	We remark that the \defref{def-kul-limitset} is valid for any group $G$ acting on a locally compact Hausdorff space $X$ and 
	Kulkarni proved the following proposition in \cite{KUL}.
	\begin{proposition}[cf.~{\cite{KUL}}] 
		Let $X$ be a locally compact Hausdorff space and $G$ be a group acting on $X$, then $G$ acts properly discontinuously on $\Omega(G)$. In addition, $\Omega(G)$ is an open subset of $X$.  Further,  if $\Omega(G)\neq \emptyset$,  then $G$ is discrete.
	\end{proposition}
	
	To classify the limit sets, we pick up Jordan forms in $\mathrm{ SL}(n, \H)$ and compute the limit set for the cyclic group generated by that Jordan form. The elements of ${\rm PSL}(n+1, \H)$ can be divided into three mutually exclusive classes as follows.
	
	\begin{definition}	\label{def:classification}
		Let $g$ be an element in ${\rm PSL}(n+1, \H)$. 
		\begin{enumerate}[(i)]
			\item $g$ is called elliptic if it is semisimple and the eigenvalue classes are represented by unit modulus quaternionic numbers. 
			
			\item $g$ is called loxodromic if it is semi-simple and not all the eigenvalue classes are represented by unit modulus quaternionic numbers.  
			
			\item $g$ is called parabolic if it is not semi-simple and eigenvalues have unit modulus quaternionic numbers.
            \item $g$ is called loxoparabolic if it is not semi-simple and at least one of the eigenvalues do not have unit modulus.
		\end{enumerate}
        \end{definition}
		These classes can be further refined.   An elliptic is rational if all of $\alpha_i$'s are rational, where $e^{2\pi i \alpha_i}$'s are representative eigenvalues of $g$. 

        \par Finally, we remark that the results in the quaternionic setting differ in several instances from those in the complex setting, as presented in \cite{cano_main}. Even when the outcomes coincide, the computational methods diverge due to the non-commutative structure of the quaternions. Notably, the computation of the first Kulkarni limit set $L_0(G)$ often deviates from the corresponding results in ${\rm PSL}(n+1, \mathbb{C})$ (cf. \cite{cano_main}).

        \par Table \ref{table} summarizes the Kulkarni limit sets for the cyclic subgroups of  ${\rm PSL}(n+1, \H)$.  
		
 \begin{table}[ht]\label{table}
\centering
\renewcommand{\arraystretch}{1.5}
\setlength{\tabcolsep}{8pt}
\begin{tabular}{||p{5.5cm}||p{4.5cm}||p{5.5cm}||}
\hline\hline
\multicolumn{3}{||c||}{\textbf{Kulkarni Limit Sets for Cyclic Subgroups of $\mathrm{PSL}(n+1,\mathbb{H})$}} \\
\hline\hline
\textbf{Conjugacy Type of Matrix} & \textbf{Jordan Form} & $\mathbf{\Lambda(G)}$ \\
\hline
\textit{Rational Elliptic}, Th. \ref{th:elliptic} & $\mathrm{D}(\lambda_1,\ldots,\lambda_{n+1})$ & $\emptyset$ \\
\hline
\textit{Simple Irrational Elliptic}, Th. \ref{th:elliptic} & $\mathrm{D}(\lambda_1,\ldots,\lambda_{n+1})$ & $\mathbb{P}^n_\mathbb{H}$ \\
\hline
\textit{Compound Elliptic}, Th. \ref{th:elliptic} & $\mathrm{D}(\lambda_1,\ldots,\lambda_{n+1})$ & $\mathbb{P}^n_\mathbb{H}$ \\
\hline
\textit{Parabolic (Type 1)}, Th. \ref{th:para1} & $\mathrm{J}(\lambda,n+1)$ & $L\{e_1,\ldots,e_n\}$ \\
\hline
\textit{Parabolic (Type 2)}, Th. \ref{th:para2} & 
$\begin{bmatrix}
\mathrm{D}(\lambda_1,\ldots,\lambda_k) & \\
 & \mathrm{J}(1,l)
\end{bmatrix}$ & 
$L\{e_1,\ldots,e_{k+l-1}\}$ \\
\hline
\textit{Parabolic (Type 3)}, Th. \ref{th:para3} & 
$\begin{bmatrix}
\mathrm{J}(1,l) & \\
 & \mathrm{J}(1,l)
\end{bmatrix}$ & 
$L\{e_1,\ldots,e_{l-1},e_{l+1}\ldots,e_{2l-1}\}$ \\
\hline
\textit{Parabolic (Type 4)}, Th. \ref{th:para4} & 
$\begin{bmatrix}
\mathrm{J}(1,k) & \\
 & \mathrm{J}(1,l)
\end{bmatrix}$ & 
$L\{ e_1,\ldots, e_{k-1},e_{k+1},\ldots, e_{k+l-1}\}$\\
\hline
\textit{Loxodromic (Type 1)}, Th. \ref{th:loxo1} & 
$\mathrm{D}(\lambda_1,\lambda_2,\ldots,\lambda_{n+1})$ & 
$L\{e_1,e_2,\ldots, e_{n}\}\cup$ $ L\{e_2,\ldots,e_{n+1}\}$ \\
\hline
\textit{Loxodromic (Type 2)}, Th. \ref{th:loxo2} & 
$\mathrm{D}(\lambda_1,\ldots,\lambda_m,\mu_{1},\ldots,\mu_p)$
& 
$L\{e_1,e_2,\ldots, e_{n}\} \cup$ $ L\{e_{m+1},e_{m+2},\ldots, e_{n+1}\}$ \\
\hline
\textit{Loxoparabolic}, Th. \ref{th:loxoparabolic} & 
$\begin{bmatrix}
\lambda_1\mathrm{J}(1,k) & \\
 & \lambda_2\mathrm{J}(1,l)
\end{bmatrix}$ & 
$L\{e_1,\ldots, e_{k-1},e_{k+1},\ldots, e_{k+l}\}\cup$  $L\{e_1,e_2,\ldots, e_{k+l-1}\}$\\
\hline\hline
\end{tabular}
\caption{{\centering Kulkarni sets for cyclic subgroups of elliptic and other elements in $\mathrm{PSL}(n+1,\mathbb{H})$}}
\label{tab:kulkarni-sets}
\end{table}
	\subsection{Structure of the paper:} After giving an introduction in the first chapter, we shall discuss some preliminaries in section \ref{sec:prelim}. In section \ref{sec:elliptic} we compute the Kulkarni sets for elliptic elements of $\mathrm{PSL}(n,\H)$ and in section \ref{sec:parabolic} we discuss about parabolic elements. In section \ref{sec:loxo} we calculate loxodromic elements. Section \ref{loxoparabolic} gives the Kulkarni limit set of the loxoparabolic element. 
    \subsection{Notations}\label{notations}
    \begin{enumerate}[(i)]
    \item  $\mathrm{D}(\lambda_1,\lambda_2,\ldots,\lambda_n)$ denotes a $n\times n$ diagonal matrix whose diagonal entries are $\lambda_1,\lambda_2,$ $\ldots,\lambda_n$.
    If $\lambda_1=\lambda_2=\ldots=\lambda_n=\lambda$ then we simply denote as $\mathrm{D}(\lambda,n)$.
        \item $ \mathrm{J}(\lambda,n+1)$ denotes a $(n+1)\times (n+1)$ matrix which have $\lambda$ at the diagonal and 1 at the super diagonal.
        \par Note that, 
        \begin{equation*}
            \mathrm{J}(\lambda,n)^m=\begin{bsmallmatrix}
        \lambda^m & {m \choose 1}\lambda^{m-1} & {m \choose 2}\lambda^{m-2} & \ldots & {m \choose n}\lambda^{m-n}\\&  \lambda^m & {m \choose 1}\lambda^{m-1} & \ldots & {m \choose n-1}\lambda^{m-n+1}\\ &   & \ddots & \ddots & \vdots\\
        &   & &\ddots & {m \choose 1}\lambda^{m-1}  \\ & &   &  & \lambda^m 
    \end{bsmallmatrix}.
        \end{equation*}
        \item $\begin{bmatrix}
            \mathrm{D}(\lambda,l) & \\ & \mathrm{J}(\mu,k)
        \end{bmatrix}$ denotes a $(l+k)\times (l+k)$ dimensional block diagonal matrix which have blocks $\mathrm{D}(\lambda,l)$ and $\mathrm{J}(\mu,k)$.
        \item   $\mathrm{Z}(m_1,n_1)$ be a $n\times n$ matrix such that all its entries except the $(m_1,n_1)^{\text{th}}$ entry is 0 and the $(m_1,n_1)^{\text{th}}$ entry is 1.
        \item $\mathbb{L}{\{p_1,p_2,\ldots,p_n\}}$ or, $\langle  p_1,p_2,\ldots,p_n \rangle $ denotes a subspce of the projective space $\P^{n}_\H$ generated by $p_1,p_2,\ldots,p_n\in \P^{n}_\H$. Similarly, $\langle U\rangle$ denotes the subspace generated by the set $U$.  
    \end{enumerate}
	\section{Preliminaries}\label{sec:prelim}
    Let \(\mathbb{H}\) denote the division ring of Hamilton’s quaternions. Every element \(a \in \mathbb{H}\) can be written as
\[
a = a_0 + a_1 i + a_2 j + a_3 k,
\]
where \(a_0, a_1, a_2, a_3 \in \mathbb{R}\), and the quaternion units satisfy \(i^2 = j^2 = k^2 = ijk = -1\). The conjugate of \(a\) is defined by
\[
\bar{a} = a_0 - a_1 i - a_2 j - a_3 k.
\]
We identify the real subspace \(\mathbb{R} \oplus \mathbb{R}i\) with the complex plane \(\mathbb{C}\). For a detailed treatment of the theory of quaternions and matrices over quaternions, see \cite{rodman} and \cite{FZ}.

\begin{definition}\label{def-eigen-M(n,H)}
Let \(A\) be an element of the algebra \(\mathrm{M}(n, \mathbb{H})\), consisting of all \(n \times n\) matrices with entries in \(\mathbb{H}\). A non-zero vector \(v \in \mathbb{H}^n\) is called a (right) eigenvector of \(A\) corresponding to the (right) eigenvalue \(\lambda \in \mathbb{H}\) if
\[
Av = v\lambda.
\]
\end{definition}

The eigenvalues of \(A\) occur in similarity classes: if \(v\) is a (right) eigenvector corresponding to \(\lambda\), then for any \(\mu \in \mathbb{H}^*\), the vector \(v\mu\) is a (right) eigenvector corresponding to the eigenvalue \(\mu^{-1} \lambda \mu\). Each similarity class contains a unique complex number with non-negative imaginary part. In what follows, we represent each similarity class by this \emph{unique complex representative} and refer to it simply as an \emph{eigenvalue}.

	Let $A \in  \mathrm{M}(n,\H)$ and write $ A = (A_1) + (A_2) j $,  where $ A_1,  A_2 \in \mathrm{M}(n,\C)$.  Now consider the embedding $ \Phi:  \mathrm{M}(n,\H)  \longrightarrow  \mathrm{M}(2n,\C)$ defined as 
	\begin{equation}\label{eq-embed-phi}
		\Phi(A) = \begin{bmatrix} A_1   &  A_2 \\
			- \overline{A_2} & \overline{A_1}  \\ 
		\end{bmatrix}. 
	\end{equation}
	
	\begin{definition}
		For $A \in  \mathrm{M}(n,\H)$, the determinant of $A$ is denoted by $ {\rm det}_{\H}(A)$  and is defined as the determinant of the corresponding matrix $ \Phi(A)$, i.e., $ {\rm det}_{\H}(A):=  {\rm det}(\Phi(A))$, see \cite{FZ}.  Because of the \textit{Skolem-Noether theorem}, the above definition is independent of the choice of the embedding chosen $ \Phi$.
	\end{definition}

	Consider the Lie groups  $\mathrm{GL}(n,\H) = \{ g \in  \mathrm{M}(n,\H) \mid {\rm det_{\H}}(g)  \neq 0 \}$ and 
	$  \mathrm{ SL}(n,\H) = \{ g \in \mathrm{GL}(n,\H) \mid {\rm det_{\H}}(g)  = 1 \}.$
	\begin{definition}[cf.~{\cite[p.  302]{he}}] \label{def-jordan}
		A Jordan block $ \mathrm{J}(\lambda,m)$ is a $m \times m$ matrix with $ \lambda $ on the diagonal entries,  1 on all of the super-diagonal entries, and zero elsewhere.  We will refer to a block diagonal matrix where each block is a Jordan block as  \textit{Jordan form}. 
	\end{definition}
	
	Next we recall the Jordan form in $\mathrm{M}(n,\H)$, see \cite[Theorem 5.5.3]{rodman}.
	
	\begin{lemma}[{Jordan  form in $ \mathrm{M}(n,\H)$, cf.~	 \cite{rodman}}] \label{lem-Jordan-M(n,H)}
		For every $A \in  \mathrm{M}(n,\H)$ there is an invertible matrix $S \in  \mathrm{GL}(n,\H)$ such that $SAS^{-1}$ has the form 
		
		\begin{equation} \label{equ-Jordan-M(n,H)}
			SAS^{-1} =  \mathrm{J}(\lambda_1, m_1) \oplus  \cdots \oplus  \mathrm{J}(\lambda_k, m_k)
		\end{equation}
		where $ \lambda_1,  \dots,  \lambda_k $ are complex numbers (not necessarily distinct) and have non-negative imaginary parts.  The forms  (\ref{equ-Jordan-M(n,H)}) are uniquely determined by $A$ up to a permutation of Jordan blocks.
	\end{lemma} 
	\subsection{The Quaternion Projective Space $\P^n_\H$}

Consider the right \(\mathbb{H}\)-module \(\mathbb{H}^{n+1}\). The \(n\)-dimensional quaternionic projective space \(\mathbb{P}^n_{\mathbb{H}}\) is defined as the quotient of \(\mathbb{H}^{n+1} \setminus \{0\}\) under the equivalence relation \(\sim\), where for \(z, w \in \mathbb{H}^{n+1} \setminus \{0\}\),
\[
z \sim w \iff z = w \, \alpha \quad \text{for some non-zero quaternion } \alpha \in \mathbb{H}.
\]
Let
\[
\mathbb{P} : \mathbb{H}^{n+1} \setminus \{0\} \longrightarrow \mathbb{P}^n_{\mathbb{H}}
\]
denote the canonical quotient map. A non-empty subset \(W \subseteq \mathbb{P}^n_{\mathbb{H}}\) is called a projective subspace of dimension \(k\) if there exists a \((k+1)\)-dimensional \(\mathbb{H}\)-linear subspace \(\widetilde{W} \subseteq \mathbb{H}^{n+1}\) such that
\[
\mathbb{P}(\widetilde{W} \setminus \{0\}) = W.
\]
Projective subspaces of dimension 1 in \(\mathbb{P}^n_{\mathbb{H}}\) are referred to as \emph{lines}.

Given a point \(p = (x_1, x_2, \ldots, x_{n+1}) \in \mathbb{H}^{n+1}\), we denote its image under \(\mathbb{P}\) by
\[
[x_1 : x_2 : \cdots : x_{n+1}] := \mathbb{P}(p).
\]
Note that for any non-zero quaternion \(\alpha \in \mathbb{H}\), we have
\[
[x_1 : x_2 : \cdots : x_{n+1}] = [x_1 \alpha : x_2 \alpha : \cdots : x_{n+1} \alpha].
\]

	Given a set of points ${S}$ in $\P^n_\H$, we define:
	$$\langle{\,S\,}\rangle = \bigcap \{H \subseteq \P^n_\H : H  \hbox { is a projective subspace containing } {S} \}.$$
	Clearly, $\langle{\,S\,}\rangle$ is a projective subspace of $\P^n_\H$.  If $p_1,p_2,\ldots,p_n$ are distinct points then $\langle p_1,p_2,\ldots,p_n\rangle$ is the unique proper quaternion projective subspace
	passing through $p_1,p_2,\ldots,p_n$.  
	
	\begin{definition}
		The complex projective space is  $\mathbb{P}^n_\mathbb{C}\subset \mathbb{P}^n_\mathbb{H}$ defined as 
        \begin{equation*}
					\P_\C^{n}=\left\{ [w]\in \P_\H^{n}:\; w=\lambda z,\; \lambda\in \H^*,\; z\in \C^{n}\right\}.
				\end{equation*}
	\end{definition}

	\subsection{Projective  Transformations}
	Action of an arbitrary element $ \gamma \in \mathrm{GL}(n+1,\H)$ on  $ \P^n_\H$ is given by:
	$$ \gamma (\P(z)) = \P( \gamma(z)) \hbox{ for all } z \in \H^{n+1}\setminus \{0\}.$$
	Note that for any non-zero  $r \in \R$,  we have $ (r\gamma) (\P(z)) = \gamma (\P(z)), \,  \forall \, z \in \H^{n+1}\setminus \{0\}.$
	We denote the corresponding quotient map by $\pi : \mathrm{GL}(n+1,\H) \longrightarrow  \mathrm{PGL}(n+1,\H)$. 
	The group of projective transformation of $\P^n_\H$ is:
	$$\mathrm{PSL}(n+1,\H) := \mathrm{SL}(n+1,\H)/ \mathcal{Z}(\mathrm{SL}(n+1,\H)),$$
	where $ \mathcal{Z}(\mathrm{SL}(n+1,\H)) = \{ \pm \mathrm{I}_{n+1}  \}$ and  it acts by the usual scalar multiplication on $\H^{n+1}$.   Given $\gamma \in \mathrm{PSL}(n+1,\H)$, we say that $ \widetilde{\gamma} \in \mathrm{SL}(n+1,\H)$  is a lift of $\gamma$ if $\pi( \widetilde{\gamma}) = \gamma$.   There is exactly two lifting for each transformation $\gamma$ in $\mathrm{PSL}(n+1,\H)$, $\tilde{\gamma}$ and $-\tilde{\gamma}$.  Note that $\mathrm{PSL}(n+1,\H)$ acts transitively on $\P^n_\H$ by taking projective subspace into projective subspace.
	
	Let $ g \in \mathrm{PSL}(n+1,\H)$ be a projective transformation,  then it is called elliptic, loxodromic and parabolic according as its lift $   \tilde{g} \in \mathrm{SL}(n,\H)$,  see Definition \ref{def:classification}.  Note that this classification is well defined in view of  Jordan decomposition of matrices over quaternions, see  \lemref{lem-Jordan-M(n,H)}. 
	\subsection{Pseudo-projective  Transformations}\label{subsec:pseudo}
We recall that the notion of pseudo-projective transformations for complex projective space was introduced by Cano and Seade in \cite{can2}, and has proven to be a useful tool in the study of complex Kleinian groups. We extend this concept to the quaternionic setting.

Let \(\widetilde{M} : \mathbb{H}^{n+1} \to \mathbb{H}^{n+1}\) be a non-zero (right) \(\mathbb{H}\)-linear transformation, which is not necessarily invertible. Let \(\ker(\widetilde{M}) \subseteq \mathbb{H}^{n+1}\) denote its kernel. Then \(\widetilde{M}\) induces a well-defined map
\[
M : \mathbb{P}^n_{\mathbb{H}} \setminus \ker(M) \longrightarrow \mathbb{P}^n_{\mathbb{H}},
\]
defined by
\[
M(\mathbb{P}(v)) = \mathbb{P}(\widetilde{M}(v)),
\]
where \(\ker(M) \subseteq \mathbb{P}^n_{\mathbb{H}}\) is the image of \(\ker(\widetilde{M}) \setminus \{0\}\) under the projective map \(\mathbb{P}\). This map \(M\) is well-defined because for all \(v \notin \ker(\widetilde{M})\), we have \(\widetilde{M}(v) \neq 0\), and projective equivalence is preserved: for any non-zero quaternion \(\alpha \in \mathbb{H}\),
\[
M(\mathbb{P}(v \alpha)) = \mathbb{P}(\widetilde{M}(v \alpha)) = \mathbb{P}(\widetilde{M}(v) \alpha) = \mathbb{P}(\widetilde{M}(v)).
\]

We refer to the map \(M\) as a \emph{pseudo-projective transformation}, and denote by \(\mathrm{QP}(n+1, \mathbb{H})\) the space of all pseudo-projective transformations of \(\mathbb{P}^n_{\mathbb{H}}\). Thus,
\[
\mathrm{QP}(n+1, \mathbb{H}) = \{ M \mid \widetilde{M} \text{ is a non-zero linear transformation of } \mathbb{H}^{n+1} \}.
\]
Note that the group \(\mathrm{PSL}(n+1, \mathbb{H})\) of quaternionic projective linear transformations is contained in \(\mathrm{QP}(n+1, \mathbb{H})\).

In what follows, we use pseudo-projective transformations, together with a classification of elements in lifts of \(\mathrm{PSL}(n+1, \mathbb{H})\) (see Definition~\ref{def:classification}), to study the limit sets of various cyclic subgroups of \(\mathrm{PSL}(n+1, \mathbb{H})\).

    \subsection{Grassmanian}
	In this article we denote Grassmanian $Gr_\H(k,n)$ as the space of all $k$-dimensional projective subspaces of $\P^n_\H$ which is compact in Hausdorff topology.

	\subsection{Some useful lemmas}
	\begin{lemma}
		\label{lem:L2}
		Let $C$ be a closed set and if the accumulation point of $\{g_n(K)\}$ lies in $L_0(G)\cup L_1(G)$ for all compact set $K\subseteq \P^n_\H\setminus C$ then $L_2(G)\subseteq C$.
	\end{lemma}
	This lemma extends the result we proved for $\mathrm{PSL}(3,\H)$ in \cite[Lemma~2.7]{dgl}.
    \subsubsection{Dual Projective Space}
    The dual of quaternionic projective space is denoted by $(\P^n_\H)^*$ are set of all hyperplanes, that is
	$$ l(\alpha_1,\alpha_2,\ldots,\alpha_{n+1})=\{[x_1:x_2:\ldots:x_{n+1}]\in \P^n_\H:\; \langle x,\alpha \rangle=\overline{\alpha_1}x_1+\overline{\alpha_2}x_2+\ldots+\overline{\alpha_{n+1}}x_{n+1}=0\}.$$
	Observe that for any $\lambda\in \H$, $\overline{\alpha_1}x_1\lambda+\overline{\alpha_2}x_2\lambda+\ldots+\overline{\alpha_{n+1}}x_{n+1}\lambda=0$ implies that hyperplanes are well defined.
	\begin{lemma} \label{lem:hyperplane}
		The hyperplanes $l(\alpha_1,\alpha_2,\ldots,\alpha_{n+1})$ and $l(\beta_1,\beta_2,\ldots,\beta_{n+1})$ are equal if and only if $[\alpha_1:\alpha_2:\ldots:\alpha_{n+1}]=[\beta_1:\beta_2:\ldots:\beta_{n+1}].$
	\end{lemma}
    \begin{proof}
Assume that \([ \alpha_1 : \alpha_2 : \ldots : \alpha_{n+1} ] = [ \beta_1 : \beta_2 : \ldots : \beta_{n+1} ]\); that is, there exists \(\lambda \in \mathbb{H}\) such that \(\beta_i = \alpha_i \lambda\) for all \(i = 1, 2, \ldots, n+1\). Let \([x_1, \ldots, x_{n+1}] \in l(\alpha_1, \ldots, \alpha_{n+1})\), i.e.,
\[
\overline{\alpha_1}x_1 + \overline{\alpha_2}x_2 + \cdots + \overline{\alpha_{n+1}}x_{n+1} = 0.
\]
Then,
\begin{align*}
\overline{\beta_1}x_1 + \overline{\beta_2}x_2 + \cdots + \overline{\beta_{n+1}}x_{n+1} 
&= \overline{\alpha_1 \lambda}x_1 + \overline{\alpha_2 \lambda}x_2 + \cdots + \overline{\alpha_{n+1} \lambda}x_{n+1} \\
&= \overline{\lambda} \left( \overline{\alpha_1}x_1 + \overline{\alpha_2}x_2 + \cdots + \overline{\alpha_{n+1}}x_{n+1} \right) = 0.
\end{align*}
Hence, \(l(\alpha_1, \ldots, \alpha_{n+1}) \subseteq l(\beta_1, \ldots, \beta_{n+1})\), and by symmetry, equality holds:
\[
l(\alpha_1, \ldots, \alpha_{n+1}) = l(\beta_1, \ldots, \beta_{n+1}).
\]

Conversely, suppose \(l(\alpha_1, \ldots, \alpha_{n+1}) = l(\beta_1, \ldots, \beta_{n+1})\), but
\([ \alpha_1 : \alpha_2 : \cdots : \alpha_{n+1} ] \neq [ \beta_1 : \beta_2 : \ldots : \beta_{n+1} ]\). There exist \(\lambda, \mu \in \mathbb{H}\) such that \(\beta_1 = \alpha_1 \lambda\), \(\beta_2 = \alpha_2 \mu\), where \(\alpha_1 \neq 0\), \(\alpha_2 \neq 0\), and \(\lambda \neq \mu\).

Now consider the vector \([1, -\overline{\alpha_1 \alpha_2^{-1}}, 0, \ldots, 0] \in \mathbb{H}^{n+1}\). We observe:
\[
\overline{\alpha_1} \cdot 1 + \overline{\alpha_2} \cdot (-\overline{\alpha_1 \alpha_2^{-1}}) = \overline{\alpha_1} - \overline{\alpha_2} \cdot \overline{\alpha_1 \alpha_2^{-1}} = \overline{\alpha_1} - \overline{\alpha_1} = 0,
\]
so this vector lies in \(l(\alpha_1, \ldots, \alpha_{n+1})\), and therefore also in \(l(\beta_1, \ldots, \beta_{n+1})\).

Then,
\[
\overline{\beta_1} - \overline{\beta_2} \cdot \overline{\alpha_1 \alpha_2^{-1}} = \overline{\alpha_1 \lambda} - \overline{\alpha_2 \mu} \cdot \overline{\alpha_1 \alpha_2^{-1}} = \overline{\lambda} \overline{\alpha_1} - \bar{\mu} \overline{\alpha_2} \cdot \overline{\alpha_1 \alpha_2^{-1}}.
\]
Using \(\overline{\alpha_2} \cdot \overline{\alpha_1 \alpha_2^{-1}} = \overline{\alpha_1}\), we get:
\[
\overline{\beta_1} - \overline{\beta_2} \cdot \overline{\alpha_1 \alpha_2^{-1}} = \overline{\lambda} \overline{\alpha_1} - \overline{\mu} \overline{\alpha_1} = (\overline{\lambda} - \overline{\mu}) \overline{\alpha_1} = 0.
\]
Since \(\overline{\alpha_1} \neq 0\), this implies \(\overline{\lambda} = \overline{\mu}\), and hence \(\lambda = \mu\), which contradicts the assumption \(\lambda \neq \mu\).

Therefore, we must have \([ \alpha_1 : \cdots : \alpha_{n+1} ] = [ \beta_1 : \cdots : \beta_{n+1} ]\), completing the proof.
\end{proof}
	Lemma \ref{lem:hyperplane} gives us that, the set of all hyperplanes in $\P^n_\H$ are in a one to one correspondence with points $x\in \P^n_\H$. We are identifying the set of hyperplanes in $\P^n_\H$ with the quaternionic projective space itself.
	\begin{lemma} \label{lem:hyper2}
		The projective transformation maps the hyperplane in $\P^n_\H$ to $\P^n_\H$ itself i.e., $g(l(\alpha_1, \alpha_2,\ldots, \alpha_n))=l(\beta_1,\beta_2,\ldots, \beta_n)$ where $\alpha=[\alpha_1:\alpha_2:\ldots:\alpha_n],\; \beta=[\beta_1:\beta_2:\ldots:\beta_n]$ and $\beta=(g^{-1})^*(\alpha)$.
	\end{lemma}
	\begin{proof}
		Consider a hyperplane defined by
        $$l(\alpha_1,\alpha_2,\ldots,\alpha_{n+1})=\{x\in \P^n_\H:\; \langle x,\alpha\rangle=0\},$$
        where the Hermitian form is given by $$\langle x,\alpha\rangle=\bar{\alpha_1}x_1+\ldots+\overline{\alpha_{n+1}}x_n=\alpha^*x.$$ 
        Let $\beta=(g^{-1})^*\alpha.$ We want to show that $g(x)\in l(\beta_1,\beta_2,\ldots,\beta_{n+1})$. Indeed,
		 $$\langle g(x),\beta\rangle=\langle g(x),(g^{-1})^*\alpha\rangle=\alpha^*g^{-1}gx=\alpha^*x=0.$$
		\par Conversely, suppose $y\in l(\beta_1,\ldots,\beta_{n+1})$ then, $g^{-1}(y)\in l(\alpha_1,\alpha_2,\ldots, \alpha_{n+1})$ as, $$ \langle g^{-1}y,\alpha\rangle=\langle g^{-1}y,g^*\beta\rangle=y^*(g^*)^{-1}g^*\beta=y^*\beta=0$$
		
		So, $g$ takes the hyperplane $l(\alpha_1,\alpha_2,\ldots, \alpha_{n+1})$ to $l(\beta_1,\beta_2,\ldots,\beta_{n+1})$, where $\beta=(g^{-1})^*\alpha$.
	\end{proof}
     The next Lemma \ref{lem:p1} is an important result of pseudo-projective transformation (cf. Sect. \ref{subsec:pseudo}).
    \begin{lemma}
			\label{lem:p1}
			$(g_n)_{n\in \N}\subset \mathrm{PSL}(n+1,\H)$ be a sequence, then there exists a sub sequence $(g_{n_k})_{k\in \N}\subset \mathrm{PSL}(n+1,\H)$ and an element $g\in \mathrm{QP}(n+1,\H)$ such that $g_{n_k}\rightarrow g$ uniformly on the compact subsets of $\P^n_\H\setminus Ker(g)$.
		\end{lemma}
		The proof of this lemma can be done by extending our result of $\mathrm{PSL}(3,\H)$ in \cite[Lemma~2.8]{dgl}. Since the generalization is elementary, we shall omit the proof of Lemma \ref{lem:p1} here.
	\section{Elliptic case}
    \label{sec:elliptic}
	In this section we observe the Kulkarni limit set for the cyclic subgroup of $\mathrm{ SL}(n+1,\H)$.
	Before proceeding to the main theorem, we shall discuss some of the useful lemmas that are needed to prove the main theorem.
	\begin{lemma}
		\label{lem:dense}
		Let $\tilde{g}\in \mathrm{PSL}(n+1,\H)$ has a lift $g=\mathrm{D}[e^{2\pi i \alpha_1},\ldots, e^{2\pi i \alpha_{n+1}}]\in \mathrm{ SL}(n+1,\H)$ 
			Then there exists a sequence $(n_k)_{k\in \N}$ such that $g^{n_k}\rightarrow I_{n+1}$ as $k\rightarrow \infty$ ($I_{n+1}$ is the identity matrix of order $n+1$). 
		\end{lemma} 
		This lemma serves as a direct extension of the result we have proven for  $\mathrm{PSL}(3,\H)$, \cite{dgl}. 
		
		\begin{theorem}\label{th:elliptic}
			Let $\tilde{g}\in \mathrm{PSL}(n+1,\H)$ be an elliptic translation given by the lift $g=\mathrm{D}[e^{2\pi i \alpha_1},\ldots, e^{2\pi i \alpha_{n+1}}]\in \mathrm{ SL}(n+1,\H)$. Then the Kulkarni limit set $\Lambda_{Kul}(G)$ for the cyclic group $G:=\langle \tilde{g} \rangle$ are as follows:
			\begin{enumerate}[(i)]
				\item Rational elliptic: If all $\alpha_i$'s are rational, then $\Lambda_{Kul}(G)=\emptyset$.
				\item Simple irrational elliptic: If all $\alpha_i$'s are equal and irrational numbers, then $\Lambda_{Kul}(G)=\P^{n}_\H$.
				\item Compound irrational elliptic: If $\alpha_i$'s are of mixed type, i.e. it can be both rational and irrational, then $\Lambda_{Kul}(G)=\P^{n}_\H$.
			\end{enumerate} 
		\end{theorem}
		\begin{proof}
			The proof of the theorem can be divided in several sub-cases depending upon the rationality of $\alpha_i$'s ($i=1,2,\ldots n$).
			\begin{enumerate}[(i)]
				\item \textbf{Rational elliptic:} Let $\alpha_i\in \mathbb{Q}$ for all $i=1,2,\ldots n$. There exists $n_0\in \N$ such that $e^{2\pi i n_0 \alpha_i}= 1$ for all $i$. As $g$ is of finite order hence $\langle g \rangle$ is.\\
				Therefore, $L_0(G)=\emptyset$ as no element have infinite isotropy group. Consequently $L_1(G)=\emptyset$ and $L_2(G)=\emptyset$ and $\Lambda_{Kul}(G)=\emptyset$.
				\item  \textbf{Simple irrational elliptic:} Let $\alpha=\alpha_1=\alpha_2=\ldots=\alpha_n$ and $\alpha\in \R\setminus \mathbb{Q}$. Recall that $zj=j\bar{z}$ for each $z\in \C$. Consider a copy of complex projective space
				\begin{equation*}
					\P_\C^{n}=\left\{ [w]\in \P_\H^{n}:\; w=\lambda z,\; \lambda\in \H^*,\; z\in \C^{n}\right\}.
				\end{equation*}
				Hence, $L_0(G)=\P_\C^{n}$.\\
				Since for any $\alpha\in \R\setminus \Q$ there exists a sequence $(n_k)_{k\in \N}$ such that $e^{2\pi i \alpha n_k}\xrightarrow{k\rightarrow \infty} 1$, \cite{kat}. Therefore, $\lim_{n_k\rightarrow\infty} g^{n_k}\begin{bmatrix}x_1\\ \vdots \\ x_{n+1}\end{bmatrix}=\begin{bmatrix}x_1\\ \vdots \\ x_{n+1}\end{bmatrix}$ for all $\begin{bmatrix}x_1\\ \vdots \\ x_{n+1}\end{bmatrix}\in \P_\H^n\setminus \P_\C^n$.
				So $\P_\H^n\setminus \P_\C^n\subseteq L_1(G)$ and hence $ L_0(G)\cup L_1(G)=\P_\H^{n}$.  Consequently, $L_2(G)=\emptyset$.
				Therefore, the Kulkarni limit set is $\Lambda_{Kul}=L_0(G)\cup L_1(G) \cup L_2(G)=\P_\H^{n}$.
				\item \textbf{Compound elliptic:} Assume that some of the $\alpha_i$'s are rational and some are irrational. In this case $g$ is a diagonal matrix of the form 
				$$g=\mathrm{D}(e^{2\pi i \alpha_1}, \ldots, e^{2\pi i \alpha_k}, e^{2\pi i \beta_1} , \ldots, e^{2\pi i \beta_l}, e^{2\pi i \gamma_1}, \ldots, e^{2\pi i \gamma_m}).$$
				Here, $\alpha_i$'s are rational numbers, $\beta_i$'s are equal irrational numbers and $\gamma_i$'s unequal irrational numbers with $\gamma_1,\gamma_2,\ldots,\gamma_p$ are rational screws i.e., there exist $n_1$ such that $\gamma_1^{n_1}=\ldots=\gamma_p^{n_1}$ and others are not.  Also, $k+l+m=n+1$.
				Let 
				\begin{equation*}
				a_1=\mathrm{D}(e^{2\pi i \alpha_1},\ldots,e^{2\pi i \alpha_k}), a_2=\mathrm{D}(e^{2\pi i \beta_1},\ldots,e^{2\pi i \beta_l}) \text{ and } a_3=\mathrm{D}(e^{2\pi i \gamma_1},\ldots,e^{2\pi i \gamma_k}).
				\end{equation*}
				
				As $\alpha_i\in \Q$, hence there exist $n_0\in \N$ such that, $e^{2\pi n_0 \alpha_i}=1$ for all $i=1,2,\ldots, k$. Assume $A_1:=\langle a_1\rangle, A_2:=\langle a_2\rangle$ and $A_3:=\langle a_3\rangle$. Therefore, as $A_1^{n_0}=I$ hence $Fix(A_1^{n_0})=L\{e_1,\ldots,e_k\}$. Also $L_0(A_2)=L_{\C}\{e_1,\ldots,e_l\}$. As $\gamma_i's$ are rational screws hence $L_0(A_3)=L_{\C}\{e_1,\ldots,e_p\}\cup \{e_{p+1},\ldots,e_m\}$. From this we can declare that $L_0(G)=L\{e_1,\ldots,e_k\}\cup L_{\C}\{e_{k+1},\ldots,e_{k+l}\}\cup L_{\C}\{e_{k+l+1},\ldots,e_{k+l+p}\}\cup \{e_{p+1},\ldots,e_m\}$.
				
				
				\medskip  Let $q\in \P_\H^n\setminus L_0(G)$. Since 
						$g^{n_k}\xrightarrow{k\rightarrow \infty} I$ hence
					$\P^2_\H\setminus L_0(G)\subseteq L_1(G)$ therefore
					 $L_0(G)\cup L_1(G)=\P^{n}_\H$.
				Consequently, $L_2(G)=\emptyset$. Hence, $\Lambda_{Kul}(G)=\P^{n}_\H$.	\end{enumerate}
			This completes the proof. 
		\end{proof}

		\section{Parabolic case}
        \label{sec:parabolic}
		In this section we discuss the Kulkarni limit sets for parabolic elements of $\mathrm{PSL}(n+1,\H)$.
		There are several cases according to the Jordan blocks.

	\begin{theorem}\label{th:para1}
			Let $\tilde{g}\in \mathrm{PSL}(n+1,\H)$ be a parabolic translation whose lift $g\in \mathrm{SL}(n+1,\H)$ is given by $g=\mathrm{J}(\lambda,n+1)$, where  $|\lambda|=1$. 
            \par The Kulkarni sets for the cyclic group $G:=\langle \tilde{g} \rangle$ are 
			$L_0(G)=\{e_1\}=L_1(G)$ and $L_2(G)=L\{e_1,e_{2}\ldots e_{n}\}$. Therefore the Kulkarni limit set would be $\Lambda_{Kul}(G)=L\{e_1,e_2,\ldots e_n\}$. 
		\end{theorem}
		\begin{proof}
			The proof can be divided in three parts each for $L_0(G),L_1(G)$ and $L_2(G)$. 
			\par Observe that, $g^m(x)=x$ for infinitely many $m$ if and only if $x=\{e_1\}$. For the calculation of $g^m$ one can refer to subsection \ref{notations}. Hence $L_0(G)=\{e_1\}.$
			\par Let us now determine $L_1(G)$. The action of $g^m$ on a projective co-ordinate of $\P^n_\H$ is given by 
            \begin{align*}
                &g^m[x_1:\cdots: x_{n+1}] \\
                &=\bigg[\lambda^m x_1+{m \choose 1}\lambda^{m-1}x_2+\cdots
                +{m \choose n}\lambda^{m-n}x_{n+1} :\lambda^m x_2+{m \choose 1}\lambda^{m-1}x_3+\cdots\\
                &+{m \choose n-1}\lambda^{m+n}x_{n+1}:\cdots: \lambda^m x_{n+1} \bigg].  
            \end{align*}
           If $x_p\neq 0$ and $x_{p+1}=x_{p+2}=\ldots=x_{n+1}=0$ then
           \begin{equation*}
               \frac{1}{{m \choose p-1}\lambda^{m-p+1}}g^{m}(x)\rightarrow e_1.
           \end{equation*}
            Similar calculations can be performed for $g^{-m}$ and in both cases the only converging point is $\{e_1\}$ and hence $L_1(G)=\{e_1\}$. \\
			\par Next consider, $\frac{1}{{m \choose n} \lambda^{m-n}}g^m$ converges to $M$ in the space of pseudo-projective transformations where $M$ has the form $\mathrm{Z}(1,n+1)$ (cf. subsection \ref{notations}). Since its $\mathrm{Ker}(M)=L\{e_1,e_2,\ldots,e_n\}$ and $\mathrm{Im}(M)=\{e_1\}\in L_0(G)\cup L_1(G)$. Thus, by Lemma \ref{lem:L2} we have $$L_2(G)\subseteq L\{e_1,e_2,\ldots, e_n\}.$$
			\par Again, if we take the hyperplane $l(1:0:\ldots:0)$ then $g^{-m}(l(1,0,\ldots,0))\rightarrow l(0,0,\ldots,1)$. As $$(g^m)^*\begin{bsmallmatrix}
				1\\0\\ \vdots\\ 0
			\end{bsmallmatrix}=\begin{bsmallmatrix}
				\overline{\lambda^m}\\ {m \choose 1}\overline{\lambda^{m-1}} \\ \vdots \\ {m \choose n}\overline{\lambda^{m-n}}
			\end{bsmallmatrix}=\begin{bsmallmatrix}
				\frac{1}{{m \choose n}}\overline{\lambda^n}\\ \frac{{m \choose 1}}{{m \choose n}}\overline{\lambda^{n-1}} \\ \vdots \\ \frac{{m \choose n}}{{m \choose n}}
			\end{bsmallmatrix}\xrightarrow{m\rightarrow\infty}e_{n+1}$$
			hence $g^{-m}(L\{e_2,\ldots,e_{n+1}\})\xrightarrow{m\rightarrow\infty} L\{e_1,e_2,\ldots,e_n\}$.
            \par Therefore $L_2(G)=L\{e_1,e_2,\ldots, e_n\}$ and thus the Kulkarni limit set is
            $$\Lambda_{Kul}(G)=L\{e_1,e_2,\ldots e_n\}.$$
            This proves the theorem. 
		\end{proof}
		\par Next, we consider elements $\tilde{g}\in\mathrm{PSL}(n+1,\H)$ whose Jordan form contains two distinct types of blocks, corresponding to different classes of eigenvalues. To distinguish these, we will employ block matrix notation. The relevant theorems are stated and proved below. The arguments naturally extend to the case where more than two types of Jordan blocks are present.
		\begin{theorem}\label{th:para2}
			Let $\tilde{\gamma}\in \mathrm{PSL}(k+l,\H)$ whose lift is $\gamma\in \mathrm{SL}(k+l,\H)$ such that
			\begin{equation*}
				\gamma=\begin{bmatrix}
				    \mathrm{D}(\lambda_1,\ldots,\lambda_k) & \\ & \mathrm{J}(1,l)
				\end{bmatrix}.
			\end{equation*}
			If $G:=\langle \tilde{\gamma} \rangle$ then $\Lambda_{Kul}(G)=L\{ e_1,e_2,\ldots, e_{k+l-1}\}$.
		\end{theorem}
		\begin{proof}
           Let \( A = \mathrm{D}(\lambda_1, \lambda_2, \ldots, \lambda_k)=\mathrm{D}(e^{2\pi i \theta_1},e^{2\pi i \theta_2},\ldots,  e^{2\pi i \theta_k}) \), and without loss of generality assume that \( \theta_1, \ldots, \theta_p \) are rational numbers, and the remaining \( k - p \) are irrational. Furthermore, let \( \lambda_{p+1}, \ldots, \lambda_q \) be irrational screws satisfying \( \lambda_{p+i} / \lambda_{p+i+1}=e^{2\pi i \alpha_i},\;\alpha_i \in \R\setminus\mathbb{Q} \). It implies \( \lambda_{p+1}^{n_0} = \cdots = \lambda_q^{n_0} \) for some \( n_0 \in \mathbb{N} \). 
			\par Therefore, $ L_0(G)=L\{ e_1,\ldots, e_p,e_{k+1}\}\cup L_\mathbb{C}\{ e_{p+1},\ldots, e_q\}\cup \left[\bigcup\limits_{i=q+1}^{k} \{e_i\}\right]$.
			\par Next, if either one of the $x_{k+2}, x_{k+3},\ldots, x_{k+l}$ is non-zero where $x=[x_1:x_2:\ldots:x_{k+l}]\in \P^n_\H$ then 
           \begin{equation*}
                \gamma^{n}(x)=\begin{bsmallmatrix}
				\lambda_1^nx_1\\ \vdots\\ \lambda_k^nx_k\\
				x_{k+1}+{n \choose 1}x_{k+2}+\ldots+{n\choose l-1}x_{k+l}\\ \vdots \\x_{k+l}
			\end{bsmallmatrix}.
           \end{equation*}
            
            By dividing $\gamma^n$ by suitable scalar it convereges to $\{e_{k+1}\}$. Also for every $y\in 
            L\{e_1,e_2,\ldots,e_{k+1}\}\setminus L_0(G)$ we can find $x\in \P^{k+l}\setminus L_0(G)$ such that $g^{n_k}(x)\rightarrow y$. This implies $$L_0(G)\cup L_1(G)=L\{e_1,\ldots,e_{k+1}\}.$$
			\par For calculating $L_2(G)$ we shall use induction on $l$.\\
			If $l=2$ then $\gamma=\begin{bmatrix}
				\mathrm{D}(\lambda_1,\ldots,\lambda_k) & \\
                 & \mathrm{J}(1,2)\end{bmatrix}$. Then $L_0(G)\cup L_1(G)=L\{ e_1,\ldots,e_{k+1}\}$ and $\gamma^n\rightarrow \tau$  in the space of pseudo projective transformations, where $\tau=\mathrm{Z}(k+1,k+2)$ (see subsection \ref{notations}). Since $Im(\tau)=\{e_{k+1}\}\subset L_0(G)\cup L_1(G)$  and $ Ker(\tau)=L\{ e_1,\ldots,e_{k+1}\}$ then we can use the Lemma \ref{lem:L2} and show $L_2(G)\subseteq L\{ e_1,\ldots,e_{k+1}\}$ by assuming $C=L\{ e_1,\ldots,e_{k+1}\}$. Hence, $\Lambda_{Kul}(G)=L\{e_1,e_2,\ldots,e_{k+1}\}$.
            \par Suppose the statement is true for $l$ then we have to prove this for $l+1$.
            \par Assume that $$\gamma=\begin{bmatrix}
				\mathrm{D}(\lambda_1,\ldots,\lambda_k) & \\
				& \mathrm{J}(1,l+1)
			\end{bmatrix}$$ then $L_0(G)\cup L_1(G)=L\{ e_1,\ldots,e_{k+1}\}$. Also, $\gamma^m\rightarrow {Z}(k+1,k+l+1)=M$ (say) and $Im(M)=\{e_{k+1}\}$, $Ker(M)=L\{e_1,e_2,\ldots,e_{k+l}\}=C$ say. By applying Lemma \ref{lem:L2} we get $L_2(G)\subseteq L\{e_1,e_2,\ldots,e_{k+l}\}$. 
            
            \par Let us define, $h_1=L\{ e_1,\ldots,e_{k+l-1}\} \text{ and } h_2=L\{ e_1,\ldots,e_{k+l}\}$. To proceed, it suffices to show that for $q\in h_1\setminus L\{ e_1,\ldots,e_k\}$ such that $L\{ q,e_{k+l}\}\subseteq \Lambda_{Kul}(\langle\gamma\rangle)$.
              By applying, the inductive hypothesis to the restriction of $\gamma$ on $h_2$, we have, $$\Lambda_{Kul}(\gamma|_{h_2})=h_1.$$ 
              Hence there exists a sequence $(q_{1m})_{m\in \N}\subseteq h_2$ such that, the cluster point of $(q_{1m})$ lies outside of 
              $L_0(G)\cup L_1(G)$ and moreover, $\gamma^m(q_{1m})\rightarrow q$. 
               Since $h_2$ is compact, we can assume $q_{1m}\rightarrow q_1$ where $$q_1=[q_1':q_2':\cdots: q_{k+l}',0] \text{ with } \sum_{j=1}^l |q'_{k+j}|^2\neq 0.$$
               Let $h_3=L\{ e_{k+2},\ldots,e_{k+l+1}\}$ then by the previous theorem we can conclude that, there exists a sequence $(q_{2m})_{m\in \N}$ such that $\gamma^m(q_{2m})\xrightarrow{m\rightarrow \infty} e_{k+l}$. Define $l_m=L\{ q_{1m},q_{2m}\}$. Therefore $\gamma^m(l_m)\xrightarrow{m\rightarrow \infty} L\{ q,e_{k+l}\}$ and hence $\Lambda_{Kul}(G)=L\{ e_1,e_2,\ldots, e_{k+l-1}\}$. 
		\end{proof}

		\begin{theorem}\label{th:para3}
			Let $\tilde{\gamma}\in \mathrm{PSL}(2l,\H)$ which have a lift $\gamma\in \mathrm{SL}(2l,\H)$ such that
			\begin{equation*}
				{\gamma}=\begin{bmatrix}
				 \mathrm{J}(1,l) & \\ & \mathrm{J}(1,l)   
				\end{bmatrix}.
			\end{equation*} 
			Then $\Lambda_{Kul}(G)= L\{ e_1,\ldots, e_{l-1},e_{l+1},\ldots, e_{2l-1}\}$ where $G:=\langle {\tilde{\gamma}}\rangle$.
		\end{theorem}
		\begin{proof}
			It is easy to prove that, $L_0(G)= L_1(G)= L\{e_1,e_{l+1}\}$.
			\par Let $K_1= L\{ e_2,\ldots,e_l\}$ and $K_2= L\{ e_{l+2},\ldots,e_{2l}\}$ be two compact sets. By the previous result we know $\gamma^n(K_1)\xrightarrow{n\rightarrow \infty} L\{ e_1,\ldots,e_{l-1}\}$ and $\gamma^n(K_2)\xrightarrow{n\rightarrow \infty} L\{ e_{l+1},\ldots,e_{2l-1}\}$. Let 
            $K=\langle K_1, K_2\rangle  $ then $K$ is $(2l-3)$-dimensional Grassmanian
			 hence $\gamma^n(K)$ converges to some $K'$. 
             From the above convergence $K'$ must be $L\{ e_1,\ldots, e_{l-1},e_{l+1},\ldots,e_{2l-1}\}.$ Hence $$L\{ e_1,\ldots, e_{l-1},e_{l+1},\ldots,e_{2l-1}\}\subseteq L_2(G).$$ 
            
            \par Furthermore, the sequence $\gamma^n$ converges to a pseudo-projective transformation  
            $$M=\begin{bmatrix}
				\mathrm{Z}(1,l) & \\ & \mathrm{Z}(1,l)
			\end{bmatrix}$$ on the domain $\P_\H^n\setminus L\{e_1,\ldots,e_{l-1},e_{l+1},\ldots, e_{2l-1}\}$ (see Lemma \ref{lem:p1}). Since the kernel of $M$ is given by $Ker(M)=L\{e_1,\ldots,e_{l-1},e_{l+1},\ldots, e_{2l-1}\}$ and image of $M$ is $Im(M)=L\{ e_1, e_{l+1}\}\subseteq L_0(G)\cup L_1(G)$ therefore using Lemma \ref{lem:L2} we have $$L_2(G)\subseteq L\{ e_1,\ldots, e_{l-1},e_{l+1},\ldots,e_{2l-1}\} .$$ This proves our claim.
		\end{proof}
     \begin{lemma}
            Let $\gamma=\mathrm{J}(1,n+1)$ and the singular value decompositions of $\gamma^m$'s are $\gamma^m=U_m D_m V_m$ where $U_m,\; V_m\in \mathrm{Sp}(n+1)$ and $D_m=\mathrm{D}(\lambda_{1m},\lambda_{2m},\ldots,\lambda_{(n+1)m})
            $. where $\lambda_{im}\geq\lambda_{(i+1)m}$ for $i=1,2,\ldots,n$ then there are $0<s<r$ such that for every $m\in \mathbb{N}$
            \begin{enumerate}[(i)]
                \item $s{m\choose n}<\lambda_{1m}<r{m\choose n}$,
                \item Also $\frac{\lambda_{2m}}{m}\longrightarrow 0$ as $m\longrightarrow \infty$.
            \end{enumerate}
            
        \end{lemma}
        The proof follows directly from its complex counterpart \cite{cano_main}.
\begin{theorem}\label{th:para4}
    Let $\tilde{\gamma}\in \mathrm{PSL}(l+k,\H)$ such that it has a lift $\gamma\in \mathrm{SL}(l+k,\H)$ where
			\begin{equation*}
				{\gamma}=\begin{bmatrix}
					\mathrm{J}(1,k) & \\ & \mathrm{J}(1,l)
				\end{bmatrix}\text{ where } k>l,
			\end{equation*} 
			then $\Lambda_{Kul}(G)= \{ e_1,\ldots,e_{k-1},e_{k+1},\ldots, e_{k+l-1}\}$ where $G:=\langle \tilde{\gamma}\rangle$.
\end{theorem}
\begin{proof}

    In the proof of this theorem we shall use the Kulkarni limit sets of the blocks $A=\mathrm{J}(1,k)$ and $B=\mathrm{J}(1,l)$.
     A direct calculation shows that $L_0(G)\cup L_1(G)=L\{e_1,e_{k+1}\}$. 

    \par Let the singular value decompositions of $A^m$ and $B^m$ are $A^m=U_{1m}D_{1m}V_{1m}$ and $B^m=U_{2m}D_{2m}V_{2m}$, where $D_{1m}=\mathrm{D}(\alpha_{1m},\alpha_{2m},\ldots, \alpha_{km})$ and $D_{2m}=\mathrm{D}(\beta_{1m},\beta_{2m},\ldots, \beta_{lm})$ then 
            \begin{enumerate}[(a)]
                \item $\frac{\alpha_{1m}}{\beta_{1m}}\longrightarrow \infty$ and
                \item $\frac{\beta_{1m}}{\max\{\alpha_{2m},\ldots,\alpha_{km},\beta_{2m},\ldots,\beta_{lm}\}}\longrightarrow \infty$.
            \end{enumerate}
            Since up to conjugacy $A$ preserve some hermitian form hence there exists $g\in \mathrm{SL}(n,\H)$ such that $\tilde{A}=gAg^{-1}\in \mathrm{PSp}(k,l)$, \cite{gms}. \\
            Also singular values of $A$ are closed under matrix inverse, i.e., if $\lambda$ is a singular value of $A$ then also $\lambda^{-1}$ is a singular value. Now dynamics of $A$ and $\tilde{A}$ are similar, so we can assume 
           $$\frac{\alpha^{-1}_{km}}{\max\{\alpha_m^{-1},\ldots,\alpha^{-1}_{(k-1)m}\}}\longrightarrow \infty.$$
      
            Without loss of generality let $U_{1m}\longrightarrow U_1,\;U_{2m}\longrightarrow U_2$ and $V_{1m}\longrightarrow V_1,\;V_{2m}\longrightarrow V_2$. \\
            Then in the space of pseudo-projective transformations $A^m=U_{1m}D_{1m}V_{1m}$ converges to $\mathrm{Z}(1,n+1)=U_1 \mathrm{D}(1,0,\ldots,0)V_1$ and hence
            $$ U_1(\langle e_1 \rangle)=\langle e_1 \rangle \text{ and }V_1^{-1}(\langle e_2,\ldots, e_k\rangle )=\langle e_1,\ldots,e_{k-1}\rangle.$$
            
            Again, $A^{-m}=V^{-1}_{1m}D^{-1}_{1m}U^{-1}_{1m}$ converges to $\mathrm{Z}(1,n+1)=V^{-1}_1 \mathrm{D}(0,0,\ldots,1)U^{-1}_1$ which gives  $$V^{-1}_1(\langle e_k\rangle)=(\langle e_1 \rangle) \text{ and } U_1(\langle e_1,\ldots,e_{k-1}\rangle)=\langle e_1,\ldots,e_{k-1}\rangle.$$
            
            Similarly, $U_2(\langle e_1\rangle)=(\langle e_1 \rangle)$, $V_2^{-1}(\langle e_2,\ldots,e_l\rangle)=\langle e_1,\ldots, e_{l-1}\rangle$ and $V_2^{-1}(\langle e_l\rangle)=(\langle e_1 \rangle)$, $U_2(\langle e_1,\ldots,e_{l-1}\rangle)=\langle e_1,\ldots, e_{l-1}\rangle$.\\
            Now $\gamma^m=U_mD_mV_m$ where $U_m=\begin{bmatrix}
                U_{1m} & \\ & U_{2m}
            \end{bmatrix}$, $V_m=\begin{bmatrix}
                V_{1m} & \\ & V_{2m}
            \end{bmatrix}$ and $D_m=\begin{bmatrix}
                D_{1m} & \\ & D_{2m}
            \end{bmatrix}$. Let $U_m\rightarrow U$ and $V_m\rightarrow V$ then $U(\langle e_1,e_{k+1}\rangle)=\langle e_1,e_{k+1}\rangle$ and $V(\langle e_1,\ldots,e_{k-1},e_{k+1},\ldots,e_{k+l-1}\rangle)=\langle e_2,\ldots,e_{k},e_{k+2},\ldots,e_{k+l}\rangle.$\\
            \par 
            Assume $C=L(e_1,\ldots,e_{k-1},e_{k+1},\ldots, e_{k+l-1})$ and $K\subseteq \mathbb{P}^{k+l-1}_\mathbb{H}\setminus C$ be a compact subset. We want to prove the accumulation point of $(\gamma^n(K))_{n\in \N}$ lies in $L_0(G)\cup L_1(G)$. Let $a\in K$ and $a_n\rightarrow a$ where $a\notin C$, Let $V_n( a_n)=\mathrm{D}(a_{1,n},\ldots,a_{k,n},a_{k+1,n},\ldots, a_{k+l,n})$ and $V(\langle a\rangle)=\mathrm{D}(a_{1},\ldots,a_{k},a_{k+1},\ldots, a_{k+l})$ then $a_1\neq 0$ or $a_{k+1}\neq 0$ . Without loss of generality, we can assume $a_{i,n}\rightarrow a_i$, $i=1,\ldots,k+l$.\\
            Now, $\begin{bmatrix}
                \alpha_{1,n} & & & &  \\ & \ddots & & & \\ & & \beta_{1,n} & & \\& & & \ddots & \\
                & & & & \beta_{l,n}\end{bmatrix}
                \begin{bmatrix}
                    a_{1,n}\\ \vdots \\ \vdots \\ \vdots \\ a_{k+l,n}
                \end{bmatrix}=\begin{bmatrix}
                    \alpha_{1,n}a_{1,n}\\ \vdots \\ \alpha_{k,n}a_{k,n} \\ \beta_{1,n}\alpha_{k,n}\\\vdots \\ \beta_{l,n}\alpha_{k+l,n}
                \end{bmatrix}=w_n.$
            Thus the accumulation point of $w_n$ lies in $(\langle e_1,e_{k+1}\rangle)$. Since $U\langle e_1,e_{k+1}\rangle=\langle e_1,e_{k+1}\rangle$ that implies the accumulation point of $\gamma_n(a_n)$ lies in $L_0(G)\cup L_1(G)$. Hence by Lemma \ref{lem:L2} we have $L_2(G)\subseteq L(e_1,\ldots,e_{k-1},e_{k+1},\ldots,e_{k+l-1}) $.
            \par Let, $L=\langle L_1 \cup L_2 \rangle$. Since  Grassmanian is compact hence $\gamma^m(L_1)\rightarrow L(e_1,e_2,\ldots, e_{k-1})$ and $\gamma^m(L_2)\rightarrow L(e_{k+1},\ldots,e_{k+l-1})$. Therefore $\gamma^m(L)\rightarrow L(e_1,e_2,\ldots,e_{l-1},e_{l+1},\ldots,e_{k+l-1})$ and hence
            $$\Lambda_{Kul}(G)=L(e_1,e_2,\ldots,e_{l-1},e_{l+1},\ldots,e_{k+l-1}).$$
            This proves the theorem. 
\end{proof}


		\section{Loxodromic case}
        \label{sec:loxo}
		We now find the Kulkarni limit sets for loxodromic elements which are defined by Definition \ref{def:classification}. 
		\begin{theorem}\label{th:loxo1}
			Let $\tilde{g}\in \mathrm{PSL}(n+1,\H)$ be a loxodromic translation given by the lift  \[ g=\mathrm{D}(\lambda_1,\ldots,\lambda_{n+1}) , \] where $|\lambda_1|<|\lambda_2|<\ldots<|\lambda_{n+1}|$. Then the Kulkarni sets for the cyclic group $G$ are given by $L_0(G)=\{e_1,e_2,\ldots, e_{n+1}\}=L_1(G)$ and
			$L_2(G)=L\{e_1,e_2,\ldots, e_{n}\}\cup L\{e_2,\ldots,e_{n+1}\}$.
			Hence, the Kulkarni limit set is $\Lambda_{Kul}(G)=L\{e_1,e_2,\ldots, e_{n}\}\cup L\{e_2,\ldots,e_{n+1}\}$ where $G:=\langle \tilde{g}\rangle.$
		\end{theorem}
		\begin{proof}
			The projective point $p=[x_1: x_2: \cdots: x_{n+1}]\in \P^n_\H$ has an infinite isotropy group if and only if $g^k(p)=p$ for infinitely many values of $k\in\Z$. Now
			$\begin{bmatrix}
				\lambda_1^k x_1:
				\lambda_2^k x_2:
				\cdots:
				\lambda_{n+1}^k x_{n+1}
			\end{bmatrix} =\begin{bmatrix}
				x_1:
				x_2:
				\cdots:
				x_{n+1}
			\end{bmatrix}$ is true only if, $p=e_1,e_2,\ldots ,e_{n+1}$. Hence, $$L_0(G)=\{e_1,e_2,\ldots , e_{n+1}\}.$$
			Now $\frac{1}{\lambda^k_{n+1}}g^{k}(p)\longrightarrow e_{n+1}$ if $x_{n+1}\neq 0$.\\
			If $x_p\neq 0$ and $x_{p+1}=x_{p+2}=\ldots=x_{n+1}=0$ with $p\neq 1$ then $g^k(x)\rightarrow \{e_p\}.$
			\par Similarly, $g^{-k}(p)\longrightarrow e_1$ if $x_1\neq 0$.
			Since, $|\lambda_1|<|\lambda_2|<\ldots <|\lambda_n|$ and as we are in a projective space therefore we can suitably multiply the vectors of $\P^n_\H$ with constants. If $x_1=x_2=\ldots =x_p=0$ and $x_{p+1}\neq 0$ then considering $\frac{1}{|\lambda_{p+1}|^k}$ we have $g^{-k}(p)\xrightarrow{k\rightarrow \infty} \{e_{p+1}\}$, for $p<n$.\\
			Therefore, $$L_1(G)=\{e_1,e_2,\ldots,e_{n+1}\}.$$
			\par  If we take the sequence $\left(\frac{1}{|\lambda_n|^k}g^k\right)_{\{k\in \N\}}$ then we have a subsequence that converges to $g_1=\mathrm{D}(0,0,\ldots,1)$ (in $g$ only last entry is 1 others are 0) uniformly on the compact subsets of $\P^n_\H\setminus Ker(g_1)$, where $Ker(g_1)=L\{e_1,e_2,\ldots, e_{n}\}$ (cf. Lemma \ref{lem:p1}). Similarly, $(g^{-k})$ has a subsequence which converges to some matrix $g_2=\mathrm{D}(1,0,\ldots,0)$ uniformly on the compact subsets of $\P^n_\H\setminus Ker(g_2)$, where $Ker(g_2)=L\{e_2,e_3,\ldots, e_{n+1}\}$.
            
            \par  If $C=L\{e_1,e_2,\ldots, e_{n}\} \cup L\{e_2,e_3,\ldots, e_{n+1}\}$ then for every compact set $K\subset \P^n_\H\setminus C$, the set of cluster points of $\{g(K)\}_{g\in G}$ is $\{e_1\}$ or $\{e_{n+1}\}$ and both are subsets of $L_0(G)\cup L_1(G)$. Hence $$L_2(G)\subseteq L\{e_1,e_2,\ldots, e_{n}\} \cup L\{e_2,e_3,\ldots, e_{n+1}\}.$$
			\par Let $[x_1:x_2:\cdots x_{n}:0]\in L\{e_1,e_2,\ldots, e_{n+1}\}$. Since $\{[y_1:y_2:\cdots:y_{n+1}]:\; y_{n+1}=y_1+y_2+\ldots +y_{n}\}$ is compact in $\P^n_\H$. Let $a_k=[\lambda_1^k x_1:\lambda_2^k x_2:\ldots: \lambda^k_{n}x_{n}:\lambda_1^kx_1+\lambda_2^k x_2+\ldots +\lambda_{n}^k x_{n} ]$. Now
			\begin{align*}
				g^{-k}(a_k) =\begin{bsmallmatrix}
					\lambda_1^{-k} &  &  & \\
					& \lambda_2^{-k} &  & \\
					& & \ddots & \\
					&  &  & \lambda_{n+1}^{-k}
				\end{bsmallmatrix}\begin{bsmallmatrix}
					\lambda_1^k x_1\\ \lambda_2^k x_2\\ \vdots\\ \lambda_{n}^kx_{n}\\
					\lambda_1^kx_1+\lambda_2^k x_2+\ldots +\lambda_{n}^k x_{n} 
				\end{bsmallmatrix}=\begin{bsmallmatrix}
					x_1 \\ x_2\\ \vdots \\ x_{n}\\ \frac{\lambda_1^k}{\lambda_{n+1}^k}x_1+\frac{\lambda_2^k}{\lambda_{n+1}^k} x_2+\ldots +\frac{\lambda_{n}^k}{\lambda_{n+1}^k} x_{n} 
				\end{bsmallmatrix}\xrightarrow{k\rightarrow \infty} \begin{bsmallmatrix}
					x_1\\x_2\\x_3\\ \vdots \\ x_{n}\\ 0
				\end{bsmallmatrix}. 
			\end{align*}
			Similarly, for $\{[x_1:x_2:\cdots: x_{n+1}]|\; x_1=x_2+ x_3\ldots +x_{n}\}$ we assume $b_k=[\lambda_2^kx_2+ \lambda_3^k x_3+\ldots + \lambda_n^k x_{n+1}:\lambda_2^kx_2: \lambda_3^k x_3:\ldots : \lambda_n^k x_{n+1}]$. Therefore $g^k(b_k)\xrightarrow{k\rightarrow \infty} \begin{bmatrix}
				0: x_2:x_3: \cdots:x_{n+1}
			\end{bmatrix}$ and $$L\{e_1,e_2,\ldots, e_{n}\} \cup L\{e_2,e_3,\ldots, e_{n+1}\} \subseteq L_2(G).$$
			 From the given Kulkarni sets we get the required $\Lambda_{Kul}(G)$. 
		\end{proof}

		\begin{theorem}\label{th:loxo2}
			Let $\tilde{g}\in \mathrm{PSL}(n+1,\H)$ be a loxodromic translation whose lift in $g\in\mathrm{SL}(n+1,\H)$ is given by 
			$$g=\mathrm{D}(\lambda_1,\ldots,\lambda_m,\mu_1,\ldots,\mu_p),$$ where $|\lambda_1|=|\lambda_2|=\ldots=|\lambda_m|<|\mu_1|<|\mu_2|<\ldots<|\mu_p|$ and $m+p=n+1$.Then 
			$$\Lambda_{Kul}(G)=L\{e_1,e_2,\ldots, e_{n}\} \cup L\{e_{m+1},e_{m+2},\ldots, e_{n+1}\}.$$
		\end{theorem}
		\begin{proof}
         Let $\lambda_1,\ldots, \lambda_r$ are real rational screws i.e. $\lambda_1^{n_0}=\lambda_2^{n_0}=\ldots=\lambda_r^{n_0}\in \R$ and $\lambda_{r+1},\ldots,\lambda_m$ are complex rational screws, i.e., $\lambda_1^{m_0}=\lambda_2^{m_0}=\ldots=\lambda_r^{m_0}\in \C\setminus \R$ for some $n_0,m_0\in \N$. This implies $L_0(G)=L\{e_1,\ldots,e_r\}\cup L_{\C}\{e_{r+1},\ldots,e_m\}\cup \{e_{m+1}\}\cup \ldots \{e_{m+p}\}$.
         Also it is easy to see
         \begin{equation*}
             L_0(G)\cup L_1(G) =L\{e_1,\ldots,e_m\}\cup \{e_{m+1}\}\cup \ldots \cup \{e_{m+p}\}.
         \end{equation*}            
			To determine $L_2(G)$ we first show that,
			$$L_2(G)\subseteq L\{e_1,\ldots, e_{n}\}\cup L\{e_{m+1},\ldots, e_{n+1}\}.$$
			If any subsequence of $g^k$ is convergent then it must converge to 
			$h_1=\mathrm{D}(0,0,\ldots,1)$ in $\P^n_\H\setminus L\{e_1,\ldots, e_{n}\}$. Also if any subsequence of $g^{-k}$ is convergent,  then it must converge to $h_2=\mathrm{D}(a_1,\ldots,a_m,0,\ldots,0)$ where $a_i\neq 0$ ($i=1,2,\ldots, m$) in $\P^n_\H\setminus L\{e_{m+1},\ldots,e_n\}$.\\
			If we assume, $C=L\{e_1,\ldots, e_{n}\}\cup L\{e_{m+1},\ldots, e_n\}$ then the result follows from Lemma \ref{lem:L2} as
			 $Im(h_1)\cup Im(h_2)\subseteq L_0(G)\cup L_1(G)$. 
            
            
			\par In the converse part, our aim is to show $L\{e_1,\ldots, e_{n}\}\cup L\{e_{m+1},\ldots, e_{n+1}\}\subseteq L_2(G)$ and we shall show this by two parts.
			
			\par Let $x=\begin{bmatrix}
				0\\ \vdots\\ 0\\ x_{m+1}\\ \vdots \\x_{n+1}
			\end{bmatrix}\in L\{e_{m+1},\ldots, e_{n+1}\}$. Construct a sequence $y_k=\begin{bsmallmatrix}
				0\\ \vdots\\  \sum_{i=1}^p \mu_i^{-k} x_{m+i}\\ \mu_1^{-k}x_{m+1}\\ \vdots\\ \vdots \\ \mu_p^{-k}x_{n+1}
			\end{bsmallmatrix}$. Note that, $y_k$ does not converege to $L_0(G)\cup L_1(G)$.
			\par Then $g^k y_k =\begin{bsmallmatrix}
				\lambda_1^k & &  &  &  &  \\
				& \ddots &  &  & &\\
				&  & \lambda_m^k &  &  & \\
				& &  &  \mu_1^k &   & \\
				& &  &  &  \ddots  & \\
				& &  &  &  & \mu_p^k
			\end{bsmallmatrix}\begin{bsmallmatrix}
				0\\ \vdots\\  \sum_{i=1}^p \mu_i^{-k} x_{m+i}\\ \mu_1^{-k}x_{m+1}\\ \vdots\\ \vdots \\ \mu_p^{-k}x_{n+1}
			\end{bsmallmatrix}=\begin{bsmallmatrix}
				0\\ \vdots\\ \frac{\lambda_m^k}{\mu_1^k}x_{m+1}+\frac{\lambda_m^k}{\mu_2^k}x_{m+2}+\vdots+\frac{\lambda_m^k}{\mu_p^k}x_{n}\\ x_{m+1}\\ \vdots \\ x_{n+1}
			\end{bsmallmatrix}$.\\
			Hence, $\lim_{k\rightarrow \infty}g^k(y_k)\longrightarrow x$ therefore $L\{e_{m+1},\ldots,e_{n+1}\}\in L_2(G)$.
			\par Next we want to show, $L\{e_1, \ldots, e_{n}\}\subseteq L_2(G)$.\\
			Let $y=\begin{bmatrix}
				x_1\\ \vdots \\ x_{n}\\0
			\end{bmatrix}\in L\{e_1,\ldots, e_{n}\}$ and construct $y_k'=\begin{bsmallmatrix}
				\lambda_1^kx_1\\\lambda_2^kx_2\\ \vdots\\ 0\\ \mu^{-k}_{p-1}x_{n-1}\\  \mu_p^{-k}x_{n}\\ \sum_{i=1}^m \lambda_i^k x_i+\sum_{j=1}^{p-1}\mu_{j}^{-k}x_{m+j}
			\end{bsmallmatrix}$.
			\par Then $g^{-k}y_k'=\begin{bsmallmatrix}
				\lambda_1^{-k} & &  &  &  & \\
				& \ddots &  & & &\\
				&  & \lambda_m^{-k} &  &  &\\
				& &  & \mu_1^{-k} &  &  \\
				& &  &   & \ddots  & \\
				& &  &  &   & \mu_p^{-k}
			\end{bsmallmatrix}\begin{bsmallmatrix}
				\lambda_1^kx_1\\\lambda_2^kx_2\\ \vdots\\ 0\\ \mu^{-k}_{p-1}x_{n-1}\\  \mu_p^{-k}x_{n}\\ \sum_{i=1}^m \lambda_i^k x_i+\sum_{j=1}^{p-1}\mu_{j}^{-k}x_{m+j}
			\end{bsmallmatrix}=\begin{bsmallmatrix}
				x_1\\x_2\\ \vdots\\ x_{n}\\ \frac{\lambda_1^k}{\mu_p^k}x_{1}+\ldots+\frac{\mu_{p-1}^k}{\mu_p^k}x_{n}
			\end{bsmallmatrix}$. Hence $\lim_{k\rightarrow \infty}g^{-k}(y'_k)=y$.
			This implies $L_2(G)=L\{e_1,\ldots, e_{n}\}\cup L\{e_{m+1},\ldots, e_{n+1}\}$. Therefore, the Kulkarni limit set is $\Lambda_{\mathrm{Kul}}(G)=L\{e_1,e_2,\ldots, e_{n}\} \cup L\{e_{m+1},\ldots, e_{n+1}\}$. \end{proof}
\section{Loxoparabolic case}\label{loxoparabolic}
In this section we shall discuss the Kulkarni sets of the loxoparabolic elements of $\mathrm{PSL}(n+1,\H)$.
\begin{theorem}\label{th:loxoparabolic}
    Let $\tilde{\gamma}\in \mathrm{PSL}(n+1,\H)$ be a loxodromic translation whose lift in $\gamma\in\mathrm{SL}(n+1,\H)$ is given by the block matrix 
			$$\gamma=\begin{bmatrix}
			    \lambda_1 \mathrm{J}(1,k) & \\ & \lambda_2\mathrm{J}(1,l)
			\end{bmatrix},\; |\lambda_1|<|\lambda_2|,\quad |\lambda_1||\lambda_2|=1.$$   Then the Kulkarni sets are given by,
			$L_0(G)=\{e_1,e_{k+1}\}=L_1(G)$ and \\$L_2(G)=L\{e_1,\ldots, e_{k-1},e_{k+1},\ldots, e_{k+l}\}\cup L\{e_1,e_2,\ldots, e_{k+l-1}\}$ where $G:=\langle \tilde{\gamma}\rangle$. Hence $\Lambda_{Kul}(G)=L\{e_1,\ldots, e_{k-1},e_{k+1},\ldots, e_{k+l}\}\cup L\{e_1,e_2,\ldots, e_{k+l-1}\}.$
\end{theorem}
\begin{proof}
It is easy to prove that $L_0(G)=L_1(G)=\{e_1,e_{k+1}\}.$
\par Assume $A_1=\mathrm{J}(1,k)$ and $A_2=\mathrm{J}(1,l)$.
    Also let $\alpha_{1n},\alpha_{2n},\ldots, \alpha_{kn}$ be the singular values of $A_1^n$ and $\beta_{1n},\beta_{2n},\ldots,\beta_{ln}$ be the singular values of $A_2^n$. Then $r {n \choose k-1}< \alpha_{1n}<s{n \choose k-1}$ where $r<s$ and $\alpha_{2n}/n\rightarrow 0$. Now $\lambda_1^nA_1^n=u_{1n}\textrm{D}(|\lambda_1|^n\alpha_{1n}, |\lambda_1|^n\alpha_{2n},\ldots, |\lambda_1|^n\alpha_{kn})v_{1n}$.
    \par Since $\lambda_1^nA_1^n\longrightarrow \mathrm{Z}(1,k)=A$ in the space of pseudo-projective transformations and without loss of generality $u_{1n}\rightarrow u_1,v_{1n}\rightarrow v_1$ also $D_{1n}\rightarrow \textrm{D}(1,0,\ldots,0)=D$ in the space of pseudo-projective transformations, then $A=uDv$. So $u_1(\langle e_1\rangle)=e_1$ and $v_1^{-1}(\langle e_2,\ldots,e_k\rangle)=\langle e_1,\ldots,e_{k-1}\rangle$ and $\lambda_1^{-n}A_1^{-n}=v_{1n}^{-1}D_{1n}^{-1}u_{1n}^{-1}$.\\
    \par Again $\lambda_1^{-n}A_1^{-n}\rightarrow \mathrm{Z}(1,k)$ and $D_{1n}^{-1}=\mathrm{D}
        (\alpha_{1n}^{-1} ,\alpha_{2n}^{-1},\ldots,\alpha_{kn}^{-1})\rightarrow \begin{bsmallmatrix}
             0 & & &\\
             & \ddots & &\\
             & & 0 & \\
             & & & 1
         \end{bsmallmatrix}$. This implies $v_1^{-1} (\langle e_k\rangle)=e_1$ and $u_1(\langle e_1,\ldots, e_{k-1}\rangle)=\langle e_1,\ldots ,e_{k-1}\rangle$. Since $v_1^{-1}(\langle e_k\rangle)=e_1$ and $v_1^{-1}(\langle e_2,\ldots, e_k\rangle)=\langle e_1,\ldots,e_{k-1}\rangle$ also $v_1$ is unitary. From this we can deduce  $v_1(e_k)=e_1$. Similarly $v_2(\langle e_l\rangle)=\langle e_1\rangle$ and $u_2(\langle e_1 \rangle)=\langle e_1\rangle.$\\
         \par Now $\gamma^n=\begin{bmatrix}
        \lambda_1^nA_1^n & \\
        & \lambda_2^nA_2^n
    \end{bmatrix}$.
    The element $\gamma^n$ can be represented as
    \begin{equation*}
        \gamma^n=\begin{bmatrix}
            u_{1n} & \\ & u_{2n}
        \end{bmatrix}\begin{bmatrix}
            |\lambda_1|^{n}\alpha_{1n} & & & & &\\
            & \ddots & & & &\\
            & & |\lambda_1|^n\alpha_{kn} & & &\\
            & & & |\lambda_2|^n\beta_1^n & &\\
            & & & & \ddots & \\
            & & & & & |\lambda_2|^n\beta^n_l
        \end{bmatrix}\begin{bmatrix}
            v_{1n} & \\ & v_{2n}
        \end{bmatrix} 
    \end{equation*}
    $u_n=\begin{bmatrix}
            u_{1n} & \\ & u_{2n}
        \end{bmatrix},v_n=\begin{bmatrix}
            v_{1n} & \\ & v_{2n}
        \end{bmatrix}$. Let $v_n\rightarrow v,\; u_n\rightarrow u$ then $v(\langle e_k\rangle)=\langle e_1\rangle, v(\langle e_{k+l}\rangle)=\langle e_{k+l}\rangle$ and $u(\langle e_1,e_2,\ldots, e_{k+l-1}\rangle)=\langle e_1,e_2,\ldots,e_{k+l-1}\rangle.$
        \par Let $C=L\{e_1,\ldots, e_{k-1},e_{k+1},\ldots,e_{k+l}\}\cup L\{e_1,\ldots, e_{l-1}\}$ and $K\subseteq \P^{k+l-1}\setminus C$ be a compact subset, we want to prove the accumulation point of $G(K)$ lie in $L_0(G)\cup L_1(G)$. 
        \par Let $a\in K$ such that $a_n \rightarrow a$, $v(a)=[a_1:\cdots:a_k:a_{k+1},\cdots: a_{k+l}]$ then $a_1\neq 0$ and $a_{k+1}\neq 0$.  Let $v_n(a_n)=[a_{1n}:a_{2n}:\ldots: a_{kn},\ldots, a_{k+l,n}]$. Then $v_n(a_n)\rightarrow v(a)$. Without loss of generality $a_{in}\rightarrow a_i$ for all $i$.
        \par Now, 
        \begin{equation*}
            \begin{bmatrix}
                |\lambda_1|^n\alpha_{1n} & & & & \\
                & \ddots & & & \\
                & & |\lambda_1|^n\alpha_{kn} & & \\
                 & & &  |\lambda_2|^n\beta_{1n} & & \\
                 & & & & \ddots & \\
                 & & & & & |\lambda_2|^n \beta_{ln}
            \end{bmatrix}
            \begin{bmatrix}
                a_{1n}\\ \vdots \\ \vdots \\ \vdots \\ \vdots \\ a_{k+l,n}
            \end{bmatrix}=\begin{bmatrix}
                |\lambda_1|^n\alpha_{1,n} a_{1,n}\\ \vdots \\ |\lambda_1|^n  \alpha_{km}a_{km}\\ |\lambda_2|^n \beta_{1n}a_{k+1,n}\\ \vdots \\ |\lambda_2|^n \beta_n a_{k+l,n}
            \end{bmatrix}=w_n.
        \end{equation*}
        Since $|\lambda_1|>|\lambda_2|$ and $a_{1n}\not\to 0$. It implies $w_n\rightarrow \{e_1\}$ since $u(\langle e_1 \rangle)=\langle e_1 \rangle$ that implies $\gamma^{n}(a_n) \rightarrow \{e_1\}$ Similarly if we compute for $\gamma^{-n}$ we can show that $\gamma^{-n}(a_n)\rightarrow \{e_{k+1}\}$ on $K\subseteq \P^{k+l-1}_\H\setminus C$. This implies $$L_2(G)\subseteq L\{e_1,\ldots, e_{k-1},e_{k+1},\ldots, e_{k+l}\}\cup L\{e_1,e_2,\ldots, e_{k+l-1}\}.$$
        \par For the converse part take $x\in L\{e_1,e_2,\ldots, e_{k+l-1}\}$ and let us assume $K=\{[x_1:\cdots:x_{k+l}] ~ | ~\; x_{k+l}=x_1+x_2+\ldots+x_{k+l-1}\}$ and take $v^{-1}(K)$ as a compact set. Take the sequence $$v_n^{-1}(k_n)=\{ |\lambda_1|^{-n}\alpha_{1n}^{-n}x_1,\ldots,|\lambda_2|^{-n}\beta_{l-1,n}^{-1}x_{k+l-1},\sum_{i=1}^{k}|\lambda_1|^{-n}\alpha_{in}^{-n}x_i+\sum_{j=1}^{l-1}|\lambda_2|^{-n}\beta_{jn}^{-n}x_j\}$$ then 
        \begin{align*}
            & D_n(k_n)=\begin{bmatrix}
                x_1\\ x_2\\ \vdots \\ x_{k+l-1}\\ \frac{|\lambda_2|^n\beta_{ln}}{|\lambda_1|^n \alpha_{1n}}x_1+\frac{|\lambda_2|^n\beta_{ln}}{|\lambda_1|^n \alpha_{2n}}x_2+\ldots+\frac{|\lambda_2|^n\beta_{ln}}{|\lambda_2|^n \beta_{l-1,n}}x_{k+l-1}
            \end{bmatrix}\rightarrow \begin{bmatrix}
                x_1\\ x_2\\ \vdots \\ x_{k+l-1}\\ 0
            \end{bmatrix}\\
            &\Rightarrow \gamma^n (v_n^{-1}(k_n))\rightarrow u(x).
        \end{align*}
        Since $u(x)\in L\{e_1,e_2,\ldots, e_{k+l-1}\} $ and $x$ is arbitrary hence $ L\{e_1,e_2,\ldots, e_{k+l-1}\}\subseteq L_2(G)$.
        \par Similarly for $x\in L\{e_1,\ldots, e_{k-1},e_{k+1},\ldots, e_{k+l}\}$ we find $\gamma^{-n}(k_n) \rightarrow x$ when $\{k_n\}\cup \{k\}\notin L_0(G)\cup L_1(G)$ where $k$ is the limit point of $(k_n)$. 
        This completes the proof.
\end{proof}
   
    \section*{Declaration of competent interest }
    The authors declare no competing interests.
    \section*{Data availability}
    No data was used for the research described in the article.
    \section*{Acknowledgments}

		Dutta acknowledges the Mizoram University for the Research and Promotion Grant F.No.A.1-1/MZU(Acad)/14/25-26. Gongopadhyay acknowledges ANRF research Grant CRG/2022/003680, and DST-JSPS Grant DST/INT/JSPS/P-323/2020. Mondal acknowledges the  CSIR grant no. 09/0947(12987)/2021-EMR-I during the course of this work.  
        The authors are also thankful to Alejandro Ucan-Puc for fruitful discussions during the course of this work.
		

\begin{thebibliography}{AAAA} 
			
			
            \bibitem[BCNS]{bcns} W. Barrera, A.  Cano, J, P. Navarrete, J. P. Seade, {\it } Elementary groups in  ${\rm PSL}(3,\C)$ in Geometry, Groups and Mathematical Philosophy, Eds: K. Gongopadhyay and S. A. Katre,   Contemporary Mathematics, Vol 811, American Mathematical Society, [Providence], RI, 2025, 31–47.
            
            
            \bibitem[CNS]{cns} Angel Cano,   Juan Pablo Navarrete,  Jos\'{e} Seade, {\it Complex {K}leinian groups},  Progress in Mathematics, Vol. 303, Birkh\"{a}user, 2013.
			
			\bibitem[CS10]{can2}   Angel Cano,  Jos\'{e} Seade, {On the equicontinuity region of discrete subgroups of {${\rm PU}(1,n)$}}, J. Geom. Anal.  20 (2010), no. 2, pp. 291--305.
			
			\bibitem[CS14]{can}   Angel Cano,  Jos\'{e} Seade, { On discrete groups of automorphisms of $\mathbb{P}^2_\mathbb{C}$}, Geom. Dedicata  168 (2014), pp.  9--60.
			
			\bibitem[CLU]{cano_main} Angel Cano, Luis Loeza, and Alejandro Ucan-Puc. Projective cyclic groups in higher dimensions. Linear Algebra and its Applications 31(5) (2017), pp. 169-209.
			
			\bibitem[CG]{cg}   S. S. Chen and L. Greenberg, \textit{Hyperbolic spaces}, {Contributions to analysis (a collection of papers dedicated to
				{L}ipman {B}ers)}, pp.  49--87, Academic Press, New York (1974).
						
			\bibitem[DGL]{dgl} S. Dutta, K. Gongopadhyay and  T. Lohan. Limit sets of cyclic quaternionic Kleinian groups. Geom Dedicata 217, 61 (2023). 

            \bibitem[GMS]{gms} K. Gongopadhyay, S. Mazumder and S. K. Sardar.  Conjugate real classes in general linear groups. Journal of Algebra and Its Applications 18(03) (2019), 1950054.
			
			\bibitem[FZ]{FZ} Fuzhen Zhang,  Quaternions and matrices of quaternions,  { Linear Algebra Appl.},  251 (1997),  pp.  21--57.
			
			\bibitem[He]{he} 	I.  N.  Herstein,  {\it Topics in Algebra},  2nd Edition,  John Wiley \& Sons,  New York, 1975.	
			
			\bibitem[KH]{kat} Anatole Katok and Boris Hasselblatt, {\it Introduction to the	Modern Theory of Dynamical Systems},  Encyclopedia of Mathematics and its Applications, Vol 54, Cambridge University Press, Cambridge, 1995.
			
			\bibitem[Mas]{Mas} Bernard Maskit, {\em Kleinian Groups}, Grundlehren der mathematischen Wissenschaften, Vol. 287, Springer-Verlag, Berlin, 1988.
			
		
			\bibitem[Ku]{KUL}
			Ravi S.  Kulkarni,  Groups with domains of discontinuity,  Mathematische Annalen 237 (1978), No. 03, pp. 253--272.
			
		
            \bibitem[Mc]{ctm}C. T. McMullen, Complex dynamics and renormalization.
            Ann. of Math. Stud., 135,
            Princeton University Press, Princeton, NJ, 1994. x+214 pp.
			
			
		
			
			
			\bibitem[Na]{TR}  J. P. Navarrete, The trace function and complex Kleinian groups in $P^2_\mathbb{C}$, International Journal of Mathematics, 19 (2008), No. 07, pp. 865--890.
			
			\bibitem[Na06]{nav}  J. P. Navarrete, On the limit set of the discrete subgroups of $\mathrm{PU}(2,1)$, Geom. Dedicata, 122 (2006), pp. 1--13.
			
			\bibitem[Ro]{rodman}
			L. Rodman.
			\newblock {\em Topics in quaternion linear algebra}.
			\newblock Princeton Series in Applied Mathematics. Princeton University Press,
			Princeton, NJ, 2014.
			
			\bibitem[SV99]{sv99} J. Seade and A. Verjovsky, Kleinian groups in higher dimensions, Rev. Semin. Iberoam. Mat. Singul. Tordesillas, 2 (1999), No. 5, pp. 1136-3894.
			
			\bibitem[SV01]{sv01} J. Seade and A. Verjovsky, {\em Actions of discrete groups on complex projective spaces}, Laminations and foliations in dynamics, geometry and topology, Contemp. Math., Vol 269 (2001), Stony {B}rook, {NY}, Amer. Math. Soc.,  pp. 155--178.
			
			
			
			
		\end{thebibliography}
	\end{document}